\documentclass[11pt]{article}
\usepackage{amssymb}
\usepackage{mathrsfs}
\usepackage{CJK}
\usepackage{bbm}
\usepackage{stmaryrd}
\usepackage{amsmath}
\usepackage{cases}
\usepackage{amscd}
\usepackage{amsfonts}
\usepackage{latexsym,amssymb,amsmath,mathrsfs,amsbsy, amsthm}
\usepackage[usenames]{color}
\usepackage{xspace,colortbl}
\usepackage[
pdfstartview=FitH, CJKbookmarks=true, bookmarksnumbered=true,
bookmarksopen=true,
citecolor=blue 
,colorlinks=true
,pdfborder=1 
,pdfstartpage=1 
,
]{hyperref} 
\allowdisplaybreaks
\usepackage{pdfsync}

\newtheorem{thm}{Theorem}[section]
\newtheorem{lem}[thm]{Lemma}
\newtheorem{prop}[thm]{Proposition}
\newtheorem{coro}[thm]{Corollary}
\newtheorem{defn}[thm]{Definition}
\theoremstyle{definition}
\newtheorem{rem}[thm]{Remark}

\newcommand{\mz}{\mathbb{Z}}
\newcommand{\mc}{\mathbb{C}}
\newcommand{\mr}{\mathbb{R}}
\newcommand{\mq}{\mathbb{Q}}

\numberwithin{equation}{section}
\begin{document}
\title{\bf On the topological decomposition of the hypersurfaces in projective toric manifolds 
 }

\author{\bf Wei WANG \\{\small School of Mathematical Sciences, Fudan University }\\
{\small Shanghai, P. R. China, 200433}}
\date{\small December, 2011}
\maketitle
\begin{abstract}
\small  In this paper, we want to discuss the topology of the non-singular hypersurface $Y^{n}$ with complex dimension $n$ in a projective toric manifold $X^{n+1}$. When $n$ is odd, our main results are a decomposition of $Y^{n}\cong Y'\sharp \ s(S^n \times S^n) $ as a connected sum of $s$ copies of $S^n \times S^n$ with a differential manifold $Y'$ such that $b_n (Y')=0$ or $2$. When $n$ is even and the degree of $Y$ in $X$ is big enough, we find that $Y$ also admits such a decomposition $Y'\sharp \ s(S^n \times S^n)$, where $Y'$ satisfy $b_n(Y')-|sign(Y')|=b_n(X)\pm sign(H^n(X))$, where $sign(H^n(X))$ is the signature of a certain bilinear form defined on $H^n(X,\mz)$.
\end{abstract}
\footnotetext[1]{\textbf{MSC(2000): 57 R 19, 57 R 65}}

\section{Introduction}
\subsection{Projective toric manifold and its hypersurfaces}
\begin{defn}
A toric variety is a normal algebraic variety $X$ containing the algebraic torus $(\mc^*)^n$ as a Zariski open subset in such a way that the normal action $(\mc^*)^n$ on itself extends to an action on $X$.

In this paper, we call $X$ a \textbf{projective toric manifold} if $X$ is a compact, smooth toric variety that admits a holomorphic embedding into a certain $\mc P^N$.
\end{defn}

The algebraic topology of projective toric manifold has been fully studied by many people and many results can be found in these two classical books \cite{Buchstaber},\cite{Fulton}. In this paper, what we need are the following two propositions (\cite{Fulton}, page 56,101,102).

\begin{prop}
Let $X$ be a projective toric manifold, then $X$ is simply connected and the odd dimension homology groups of $X$ vanish, i.e. $H_{odd}(X,\mz)=0$.
\end{prop}

\begin{prop}
$H_*(X,\mz)$ can be generated by the projective toric submanifolds of $X$, i.e. there exist smooth toric submanifolds $\{X_i\}$ with $x_i=[X_i]\in H_*(X,\mz)$ such that the homomorphism
\[
\sum \mz x_i \longrightarrow H_*(X,\mz)
\]
is surjective.
\end{prop}

Then we introduce the hypersurface of a projective toric manifold. Let $X$ be a projective toric manifold. For any holomorphic embedding $X\hookrightarrow \mc P^N$, let $F_1$ be a hyperplane of $\mc P^N$, we get a subvariety $i: Y= F_1 \cap X\hookrightarrow X$ of $X$ and $Y$ is called a \textbf{hypersurface} of $X$. By Bertini's theorem, for a generic hyperplane $F_1$ in $\mc P^N$, $Y$ is smooth.

Given such a hypersurface $Y$ of $X\hookrightarrow \mc P^N$, we can also construct the smooth hypersurface $i_d:Y_d\hookrightarrow X$ of $X$ with $(i_d)_*[Y_d]=d(i_* [Y]),\ 0<d\in \mz$. Indeed, we can take $Y_d:=F_d\cap X$, where $F_d$ is a generic hypersurface of $\mc P^N$ with degree $d$ and it is well-known that $Y_d$ is also a smooth hypersurface of $X$.

In this paper, all the hypersurfaces we consider are smooth and when we say a hypersurface $Y_d$, it  always means $Y_d$ is a smooth hypersurface.

Similar to the degree of a hypersurface in $\mc P^n$, we can define the degree of a (smooth) hypersurface in $X$. Let $\alpha_Y$ be the element of $H^2(X,\mz)$ such that $\alpha_Y\cap [X]=i_*[Y]$. We define the \textbf{degree} of a hypersurface $Y$ in $X$ by
\[
degY:= <\alpha_{Y}^{n+1},[X]>
\]
For the hypersurface $Y_d$, we have relation $d\alpha_Y=\alpha_{Y_d}$ and we have $degY_d=d^{n+1}degY$.

\subsection{Main results}
Let $X^{n+1}$ be a projective toric manifold with complex dimension $n+1$, $n>2$. Let $i:Y\hookrightarrow X$ be the hypersurface of X with complex dimension $n$ and $i_d:Y_d\hookrightarrow X$ be the hypersurface with $(i_d)_*[Y_d]=d(i_* [Y]),\ 0<d\in \mz$. In this paper, we want to discuss the topological decomposition of the hypersurface $Y_d$. Our main results are:
\begin{thm}
When $n$ is odd, for any integer $d>0$, we have decomposition:
\[
Y_d \cong Y_{d}^{'}\ \sharp \ s_d (S^n\times S^n)
\]
where the n-th Betti number $b_n(Y_{d}^{'})=0$ or $2$.
\end{thm}

\begin{thm}
When $n$ is even, for sufficiently big $d>>0$, we have decomposition:
\[
Y_d \cong Y_{d}^{'}\ \sharp \ s_d (S^n\times S^n)
\]
with $s_d = \frac{b_n(Y_d)-b_n(X)-|sign(Y_d)-sign(H^n(X))|}{2}$, here $sign(Y_d)$ is the classical signature of $Y_d$ and $sign(H^n (X))$ is the signature of the bilinear form defined by:
\[
H^n(X)\otimes H^n(X)\longrightarrow \mz
\]
\[
(x,y)\mapsto <x\cup y\cup \alpha_{Y_d},[X]>
\]
Furthermore, we have limit estimate:
\[
0<\lim_{d\rightarrow +\infty}\frac{2s_d}{degY_d}=\lim_{d\rightarrow +\infty}\frac{2s_d}{d^{n+1}degY}=1- 2^{n+1}(2^{n+1}-1)\frac{B_{\frac{n+2}{2}}}{(n+1)!}<1
\]
here $B_{\frac{n+2}{2}}$ is the $\frac{n+2}{2}$-th Bernouli number.
\end{thm}

For $Y_d'$, we have relations $b_n (Y_d')=b_n (Y_d)- 2s_d$ and $sign(Y_d)=sign(Y_d')$. By the limit estimates in proposition 4.4, we know $|sign(Y_d)|=|sign(Y_{d}^{'})|$ tends to $+\infty$ as $d\rightarrow +\infty$. From theorem 1.5, we can deduce that
\begin{coro} When $n$ is even and $d$ is big enough:
\[
b_n(Y_{d}') - |sign(Y_{d}^{'})|= b_n( X) \pm |sign(H^n(X))|
\]
\end{coro}

\begin{rem}
Let $F$ be a nonsingular algebraic hypersurface in complex projective space, in \cite{KW}, Kulkarni and Wood proved that there is a differentiable connected sum decomposition
\[
F=M \sharp k (S^n \times S^n)
\]
where $b_n (M) = 0 $ or $2$ for $n$ odd, and $b_n (M) -|sign(M)|=b_n(\mc P^n )\pm sign(H^n (\mc P^{n+1}))=0$ or $2$ for $n$ even.

Our theorem is a generalization of their theorem to the case of hypersurfaces in projective toric manifolds.
\end{rem}

\section{Basic idea of removing handles}
\subsection{Geometric point of view}
Choose a point $(x,y)\in S^n \times S^n$ and there are two embedded spheres: $S_1:=S^n \times \{y\},\ S_2:= \{x\}\times S^n \hookrightarrow S^n \times S^n$ with properties:\\
(1). $S_1$ intersects $S_2$ transversally at one point $(x,y)$.\\
(2). The normal bundles of $S_1$, $S_2$ in $S^n \times S^n$ are trivial.\\
(3). Denote $\eta_1 := S_1 \times D^n\subset S^n \times S^n $ and $\eta_2 := S_2 \times D^n\subset S^n\times S^n$ by the closure of their normal bundles, we see $\eta_1 \cup \eta_2$ is a manifold with boundary $S^{2n-1}$ and
\[
S^n\times S^n = (\eta_1 \cup \eta_2) \cup_{S^{2n-1}} D^{2n}
\]
Conversely, let $M^{2n}$ be a smooth manifold and $S_1,\ S_2$ be two embedded n-spheres of $M^{2n}$ with:\\
(1). $S_1$ intersects $S_2$ transversally at one point.\\
(2). The normal bundles of $S_1$ and $S_2$ are trivial.\\
We denote $\xi_{1}:= S_1\times D^n, \ \xi_{2}:= S_2 \times D^n$ by the closure of their normal bundles. Observe that $\eta_1 \cup \eta_2 \cong \xi_{1}\cup \xi_{2}$ and we get:
\[
M\cong (M-\xi_{1}^{\circ}\cup \xi_{2}^{\circ})\cup_{S^{2n-1}}(\eta_1\cup\eta_2)\cong M' \sharp\  S^n\times S^n
\]
where $M' = (M-\xi_{1}^{\circ}\cup \xi_{2}^{\circ})\cup_{S^{2n-1}} D^{2n}$. This is the basic idea of removing handles from a $2n$-manifold (\cite{Wood1}). Next we want to realize this idea by algebraic topology.

\subsection{Homological point of view}
From the point of view of homology, let $M^{2n}$ be a simply connected smooth closed manifold of dimension $2n$, $n>2$ and $h: \pi_n (M)\longrightarrow H_n(M,\mz)$ be the Hurewicz map. For every $\alpha, \ \beta \in h(\pi_n(M))\subset H_n(M,\mz)$ with intersection number $\alpha\cdot\beta=1$, by Whitney's embedding theory and Whitney's trick (\cite{Ranichi}, p142), there are two embedding $n$-spheres $f_{\alpha},\ f_{\beta}: S^n \hookrightarrow M^{2n}$ with:\\
(1). The homology elements $\alpha$ and $\beta$ are represented by $f_{\alpha},\ f_{\beta}$, i.e. $(f_{\alpha})_* [S^n]=\alpha, \ (f_{\beta})_* [S^n]=\beta $\\
(2). The spheres $f_{\alpha}(S^n)$ and $f_{\beta}(S^n)$ intersect transversally at only one point.

Following the geometric idea of removing handles, the next question is how to determine the normal bundles. In general, the normal bundles of $f_{\alpha}(S^n),\ f_{\beta}(S^n)$ are not easy to determine.

In this paper, the situation seems relatively simpler: let $K\subset h(\pi_n (M))$ be a free Abel group such that each element $\alpha\in K$ can be represented by an embedded $n$-sphere $f_{\alpha}:S^n\hookrightarrow M$ with \textbf{stable trivial} normal bundle.

When $n$ is even, for the embedding $f_{\alpha}$ representing $\alpha\in K$, the normal bundle of $f_{\alpha}$ is just determined by the self-intersection number $\alpha\cdot\alpha$ of $\alpha$. Indeed, $\alpha\cdot\alpha=0$ if and only if the normal bundle of $f_{\alpha}$ is trivial. So, if we could find a free subgroup $\oplus_{i=1}^{s}(\mz \alpha_i \oplus \beta_i)$ of $K$ with intersection matrix $
\oplus_{i=1}^{s} \left(\begin{array}{cc}
0 & 1\\
1 & 0
\end{array}
\right)$, topologically, $M$ admits a decomposition:
\[
M\cong M' \sharp \ s(S^n \times S^n)
\]

When $n$ is odd, the intersection number $\alpha\cdot\alpha$ is always zero and can not determine the normal bundle of $f_{\alpha}$, we need two techniques.

\textbf{Technique 1:} find a quadratic function $\psi:K\longrightarrow \mz_2$ with:\\
(1). $\psi(\alpha+ \beta)=\psi(\alpha)+\psi(\beta)+ (\alpha\cdot\beta)_2$, where $(\alpha\cdot\beta)_2\in \mz_2$ is the mod 2 class of the intersection number $\alpha\cdot\beta\in \mz$, which is also the definition of the quadratic function over $\mz$.\\
(2). $\psi(\alpha)=0$ if and only if $\alpha$ can be represented by an embedded $n$-sphere $f_{\alpha}$ with trivial normal bundle.

For any free subgroup $\oplus_{i=0}^{s}(\mz \alpha_i' \oplus \beta_i')$ of $K$ with intersection matrix $
\oplus_{i=0}^{s} \left(\begin{array}{cc}
0 & -1\\
1 & 0
\end{array}
\right)$, by the standard results of the quadratic function (\cite{Wall}, p172), we can find a new basis $\{\alpha_i,\beta_i\}$ of this subgroup such that the intersection matrix of
$\mz\alpha_0 \oplus \mz\beta_0 \oplus \{\oplus_{i=1}^{s}(\mz \alpha_i \oplus \beta_i)\}$ is still $
\oplus_{i=0}^{s} \left(\begin{array}{cc}
0 & -1\\
1 & 0
\end{array}
\right)$
and $\psi(\alpha_i)=\psi(\beta_i)=0, i\neq 0, \ \psi(\alpha_0)=\psi(\beta_0)=0$ or $1$. In this case, although we can not determine the value of $\psi(\alpha_0)$, at least, we have decomposition:
\[
M\cong M' \sharp \ s(S^n \times S^n)
\]

In general, the quadratic function $\psi$ is not always exist on $K$ and we need the second technique.

\textbf{Technique 2:} find an embedded $n$-sphere $g:S^n \hookrightarrow M$ with:\\
(1). $g_* [S^n]=0\in H_n(M,\mz).$\\
(2). The normal bundle $\eta_g$ of $g:S^n \hookrightarrow M$ is isomorphic to the tangent bundle $TS^n$ of $S^n$.

If we could find such an embedding $g$, for every element $\alpha\in K$ which is represented by an embedding $f_{\alpha}$, if the stable trivial normal bundle $\eta_{f_{\alpha}}$ is not trivial, by Wall's technique (\cite{Wall}, p167), there exists a new embedding $f_{\alpha}'$ with normal bundle $\eta_{f_{\alpha}'}$ such that:\\
(1). $f_{\alpha}'=f_{\alpha}+g\in \pi_n (M)$\\
(2). $F(\eta_{f_{\alpha}'})=F(\eta_{f_{\alpha}}) + F(\eta_g)$
where $F$ is the isomorphism:
\[
\{ \text{$n$ dimensional stable trivial vector bundles over $S^n$}\} \longleftrightarrow Ker(\pi_{n-1}(SO(n))\rightarrow \pi_{n-1}(SO))
\]
It is known that $Ker(\pi_{n-1}(SO(n))\rightarrow \pi_{n-1}(SO))=0, \ n=1,3,7$ and $Ker(\pi_{n-1}(SO(n))\rightarrow \pi_{n-1}(SO))=\mz_2$ with generator $F(TS^n)$, $n$ odd $\neq 1,3,7$. (\cite{BrowderS}, p88).

In this case, modifying by this embedding $g,$ we can make every element $\alpha\in K$ represented by an embedding with trivial normal bundle. So, for any free subgroup $\oplus_{i=1}^{s}(\mz \alpha_i \oplus \beta_i)$ of $K$ with intersection matrix $
\oplus_{i=1}^{s} \left(\begin{array}{cc}
0 & -1\\
1 & 0
\end{array}
\right)$, $M$ admits a decomposition:
\[
M\cong M' \sharp \ s(S^n \times S^n)
\]

These are the basic tools to removing ($n$ dimensional) handles from a $2n$-manifold. In the next two sections, we will apply these tools to prove our main theorems.
\section{Odd case}
\subsection{Wu class, quadratic function, and Kervaire invariant}
Given a smooth manifold $(M^n,\partial M)$, the Steenrod operator $Sq=\sum_{i=0} {Sq}^i:H^*(M,\partial M,\mz_2)\rightarrow H^*(M,\partial M,\mz_2)$ determines a linear form on $H^*(M,\partial M,\mz_2)$:
\[
H^*(M,\partial M )\longrightarrow {\mz}_2
\]
\[
x \mapsto <Sq(x),[M]>
\]
where $[M]\in H_n(M,\partial M,\mz_2)$ is the fundamental class of the Poincar\'{e} pair $(M,\partial
M)$. Since the cup product induces the isomorphism $H^*(M,\mz_2)\cong Hom(H^*(M,\partial M,\mz_2),{\mz}_2)$, there exists a unique element $v(M)=1+v_1(M)+v_2(M)+\cdot\cdot\cdot\in H^*(M,\mz_2) $ such that for each $x\in H^*(M,\partial M,\mz_2)$:
\[
<v(M)\cup x,[M]>= <Sq(x),[M]>,\  <v_i(M)\cup x,[M]>=<Sq^i x,[M] >
\]
\begin{defn}
$v(M)=\sum_{i=0}v_i (M)$ is called the Wu class of $M$.
\end{defn}
By the definition, we see $v_i(M)=0 \Longleftrightarrow Sq^i:H^{n-i}(M,\partial M,\mz_2)\rightarrow \mz_2$ is zero.

In his paper \cite{BrowderC}, Browder gave a geometric definition of Kervaire invariant, which is equivalent to his original definition of Kervaire invariant in \cite{BrowderK}. This geometric definition is very close to the original definition of Kervaire, which is defined by the Arf invariant of a certain quadratic function, (cf \cite{KervaireMilnor}). First, it is known in \cite{BrowderC} that:
\begin{prop}
For any $x\in H_n (M^{2n},\mz_2)$, we can find an embedded $N^n \subset M^{2n}$ with $[N]=x$.
\end{prop}
\begin{prop}
If $M^{2n}\times {\mr}^q\subset W^{2n+q}$, $W$ connected and $y\in H_{n+1}(W,M,\mz_2)$, we can find $N\subset M $ representing $\partial y\in H_n(M,\mz_2)$, i.e $[N]=\partial y$ with $N=\partial V$, here $i:V\subset W\times [0,1]$ is a connected submanifold with $i_* [V]=y$, where $[V]\in H_{n+1}(V,\partial V,\mz_2)$ is the fundamental class. Furthermore, $V$ meets $W\times 0$ transversally in $N\subset M$.
\end{prop}

We see, in this case, the normal bundle of $N$ in $W\times 0$ has a normal $q$-frame ($N\times {\mr}^q \subset M\times {\mr}^q\subset W$). The obstruction to extending this frame to a normal $q$-frame on $V\subset W\times [0,1]$ lies in $H^{i+1}(V,N,\pi_{i}(V_{n+q,q}))$ and $\pi_{i}(V_{n+q,q})=0,q<n, \ \pi_{n}(V_{n+q,q})=\mz_2$. We find the last and only one obstruction $\sigma\in H^{n+1}(V,N,\pi_{n}(V_{n+q,q}))=\mz_2$.
\begin{defn}
For the element $x=\partial y\in H_n (M,\mz_2)$, where $\partial: H_{n+1}(W,M,\mz_2)\longrightarrow H_n (M,\mz_2)$ and $y\in H_{n+1}(W,M,\mz_2)$, we define: $\psi(x)=<\sigma,[V]>,$ which is denoted briefly by $\sigma$ for convenience.
\end{defn}

We see this definition seems not intrinsic, it depends on the choice of $N$ and $V$. We should put some condition to make $\psi$ well-defined. Browder proved (\cite{BrowderC}):

\begin{prop}
The obstruction to extend a $q$-frame defines a quadratic form:
\[
\psi: Ker(H_n (M,\mz_2)\longrightarrow H_n(W,\mz_2))\rightarrow \mz_2
\]
if and only if $v_{n+1}(W)=0$.
\end{prop}

\begin{prop}
For the embedding $\phi(S^n)\in M^{2n},n$ odd, if $\phi(S^n)$ is nullhomotopic in $W$, then $\psi([\phi(S^n)])=0$ if and only if the normal bundle of $\phi(S^n)$ is trivial.
\end{prop}

\begin{defn}
If Ker$(H_n (M,\mz_2)\longrightarrow H_n(W,\mz_2))$ is non-singular under the intersection pair, we define the Kervaire invariant $k$ by its Arf invariant of the quadratic form $\psi$.
\end{defn}

\subsection{Proof of the odd case I}
Let $X^{n+1}$ be a projective toric manifold with complex dimension $n>2,$ odd, and $i:Y^n\hookrightarrow X^{n+1}$ be a hypersurface of $X^{n+1}$.
\begin{lem}
$H_n(Y,\mz)$ is spherical and every element $\alpha\in H_n(Y,\mz)$ can be represented by an embedding $f_{\alpha}:S^n \hookrightarrow Y$ such that the normal bundle $\eta_{f_{\alpha}}$ of $f_{\alpha}$ is stable trivial.
\end{lem}
\begin{proof}
First, by Lefschetz's hyperplane section theorem and Proposition 1.2., we know $(X,Y)$ is $n$-connected and $H_n(X)=0$. We have exact sequence:
\[
\begin{CD}
H_{n+1}(X,Y)@>>> H_n (Y) @>>> H_n(X)=0\\
@AA\cong A        @AA h_Y A     @AAA \\
\pi_{n+1} (X,Y) @>>> \pi_n (Y) @>\pi_n(i)>> \pi_n (X)
\end{CD}
\]
From this diagram, we observe that $h_Y:\pi_n (Y)\longrightarrow H_n (Y)$ is surjective and for every element $\alpha\in H_n(Y)$, by the Whitney embedding theorem, we can choose an embedding $f_{\alpha}:S^n \hookrightarrow Y$ to represent $\alpha$ such that $i\circ f_{\alpha}$ is nullhomotopic in $X$, i.e. $\pi_n(i)([f_{\alpha}])=0$.

Second, we want to show the normal bundle $\eta_{f_{\alpha}}$ is stable trivial. We have bundle identity:
\[
TX|S^n = (i\circ f_{\alpha})^* TX = TS^n \oplus \eta_{f_{\alpha}} \oplus \eta_{Y}^{X}| S^n
\]
here $\eta_{Y}^{X}$ is the normal bundle of $i:Y \longrightarrow X$. Since $\eta_{Y}^{X}$ is a complex line bundle, it is known that $\eta_{Y}^{X}\cong i^*L_Y$, where $L_Y$ is a complex line bundle over $X$ with Euler class $e(L_Y)\cap [X]=i_* [Y]$.

Since $i\circ f_{\alpha}$ is nullhomotopic, the bundle identity becomes:
\[
\epsilon^{2n+2}=(i\circ f_{\alpha})^* TX = TS^n \oplus \eta_{f_{\alpha}} \oplus (i\circ f_{\alpha})^*L_Y =\epsilon^{n+1}\oplus \eta_{f_{\alpha}}
\]
here $\epsilon$ is the trivial real 1-bundle.
\end{proof}

\textbf{\textsl{Proof of the odd case I:}} For the complex line bundle $L_Y$ in the above lemma, consider $W=D(-L_Y)$, where $-L_Y$ is the stable inverse bundle of $L_Y$, i.e. $L_Y \oplus -L_Y$ is trivial, and $D(-L_Y)$ is the disk bundle of $-L_Y$. Then for the embedding: $Y\hookrightarrow X\hookrightarrow W$, we see the normal bundle of $Y$ in $W$ is trivial and we get $Y\times \mr^q \subset W$ for some $q>0$.

Observe that $Ker(H_n(Y,\mz_2)\longrightarrow H_n(W,\mz_2))=H_n(Y,\mz_2)$ and by proposition 3.5, if the Wu class $v_{n+1}(W)=0$, there is a quadratic function $\psi':H_n(Y,\mz_2)\longrightarrow \mz_2$ and we also obtain a quadratic function on $H_n (Y,\mz)$:
\[
\begin{CD}
\psi: H_n(Y,\mz)@>>> H_n(Y,\mz_2)@>\psi'>> \mz_2
\end{CD}
\]
Furthermore, by proposition 3.6, we know $\psi(\alpha)=0$ if and only if the normal bundle $\eta_{f_{\alpha}}$ is trivial.

Since $H_n(Y,\mz)$ is unimodular, by technique 2 in subsection 2.2, $H_n (Y,\mz)\cong \mz\alpha_0 \oplus \mz\beta_0 \oplus \oplus_{i=1}^{s}(\mz \alpha_i \oplus \beta_i),\ s=\frac{b_n(Y)-2}{2}$ with intersection matrix $
\oplus_{i=0}^{s} \left(\begin{array}{cc}
0 & -1\\
1 & 0
\end{array}
\right)$
and $\psi(\alpha_i)=\psi(\beta_i)=0, i\neq 0.$ Topologically, we get decomposition:
\[
Y\cong Y' \sharp \ s (S^n \times S^n)
\]
where $b_n (Y')=2 $. If the Kervaire (Arf) invariant $k$ of $\psi$ or $\psi'$ vanishes, we can make $b_n(Y')=0$.

So we finish the proof of the odd case when $v_{n+1}(W)=0$. When $v_{n+1}(W)\neq 0,$ the quadratic function $\psi$ is not necessary well-defined and we will use technique 2 to deal with it in next subsection.

\subsection{Proof of the odd case II}
In his paper \cite{BrowderC}, Browder proved:
\begin{thm}[Browder] Suppose $M^{2n}\times \mathbb{R}^q \subset W , \ n\neq 1,3$ or $7$, $W$ is 1-connected. $(W,M)$ is n-connected and suppose $v_{n+1}(W)\neq 0$. Then there exists an embedded $S^n \subset M^{2n}$ and $U^{n+1}\subset M^{2n}\times \mathbb{R}^{q+1}$ with $\partial U=S^n$ such that the normal bundle $\xi$ to $S^n$ in $M^{2n}$ is nontrivial, but $\xi\oplus\epsilon^1$ is trivial, where $\epsilon^1$ is the trivial one dimensional real vector bundle. Hence $S^n$ is homologically trivial $(mod \ 2)$ with nontrivial normal bundle.
\end{thm}

\begin{rem}
It seems we can use this theorem to find the embedding sphere in technique 2. But the shortage is: the embedding sphere $S^n$ is only mod 2 trivial. We want to add some condition to make it work in integral  homology.
\end{rem}

\begin{thm}
Under the same hypothesis of Browder's theorem above, if we further assume that $H_{n+1}(W,\mz_2)$ is generated by the element $\{x_i\}$ such that each $x_i$ can be represented by an oriented closed manifold $N_i$, i.e $[N_i]=x_i$. Then there exist an embedded $S^n\subset M^{2n}$ such that $[S^n]=0\in H_n(M,\mz)$ and the normal bundle $\xi$ to $S^n$ in $M^{2n}$ is non-trivial but stable trivial.
\end{thm}
\begin{proof}
We follow Browder's proof (Step 2 to Step 7 is almost unchanged):\\
Step 1: Since $v_{n+1}(W)\neq 0,$ we know $Sq^{n+1}:H^{n+q-1}(W,\partial W,\mz_2)\rightarrow \mz_2$ is not zero. By assumption, $\exists N_i$ such that $Sq^{n+1}y_i\neq0$, where $y_i\cap [W]=[N_i]$.\\
Step 2: For convenience, we denote $N_i$ by $N$ and $y_i$ by $y$. Let $N_0 =N-int D^{n+1}$, we see $\partial N_0 =S^{n+1}$ and $N_0$ is homotopic to an $n$-complex. Since $(W,M)$ is $n$-connected, $\exists f:N_0 \longrightarrow M$ such that the diagram is commutative up to homotopy:
\[
\begin{CD}
N_0 @>f>> M\\
@VVV      @VVV\\
N @>>>    W \\
\end{CD}
\]
Step 3: Let $g=f|_{\partial N_0}:S^n \rightarrow M,$ we see $g_{*}[S^n]=0\in H_n(M,\mz)$. Since $M$ is 1-connected, by Whitney's theorem, we can make $g$ homotopic to an embedding and we still denote it by $g$. Since $g$ is nullhomotopic in $W$, then the normal bundle of $g(S^n)$ is stable trivial. We wish to show that the normal bundle of this sphere is not trivial.\\
Step 4: We make the map $f: N_0 \longrightarrow M \times \mr^q\times [-1,0]$ homotopic to an embedding $g_0: N_0\hookrightarrow M\times \mr^q \times [-1,0]$ such that $g_0 |_{\partial N_0}=g$. And we extend $g:S^n \hookrightarrow M$ to $\tilde{g}: D^{n+1}\subset W\times [0,1]$ which meets $W\times 0$ transversally in $g(S^n)$.
Then we get an embedding $g_1 : N_0 \cup_{S^n} D^{n+1}\cong N \hookrightarrow W\times [-1,1],$ which is isotopic to the origin $N\subset W$.\\
Step 5: $M\times \mr^q\subset W\times 0$ define a $q$-frame of the normal bundle of $g(S^n)\subset M\subset W\times 0$. We know the obstruction $\sigma\in \pi_{n+1}(V_{n+1,q})$ to extend this $q$-frame to $D^{n+1}$ is zero if and only if the normal bundle of $g(S^n)$ is trivial.\\
Step 6: Now assume the normal bundle of $g(S^n)$ is trivial and we get a $q$-frame on $D^{n+1}$:
\[
D^{n+1}\times D^n \times \mr^q \subset W \times [0,1]
\]
such that $D^{n+1}\times 0 \times 0 =\tilde{g}(D)$ and $S^n \times D^n \times 0 $ is the normal bundle of $g(S^n)$.
Let $V=M\times[-1,0]\cup_{S^{n}\times D^n} D^{n+1}\times D^n$, then we get $V\times \mr^q \subset W\times [-1,1]$, $g_1(N)\subset int V\times \mr^q$.\\
Step 7: Let $Y=W\times [-1,1]/\partial (W\times [-1,1])$ we get:
\[
\begin{CD}
Y @>a>> \Sigma^q V/\partial V @>b>> T(\eta_N \oplus \epsilon)
\end{CD}
\]
where $\eta_N$ is the normal bundle of $N$ in $W$. Let $U$ be the mod 2 Thom class of $\eta_N\oplus \epsilon$, we get:
$(ba)^* U = \Sigma x\in H^{n+q}(Y,\mz_2)$ and $(Sq^{n+1}(x))[W]=Sq^{n+1}(\Sigma x)[Y]\neq 0$. Also, $(Sq^{n+1}(b^*U))(\Sigma^q [V])\neq 0$ and $Sq^{n+1}(\Sigma^{-q}(b^* U))[V]\neq 0$. On the other hand, $Sq^{n+1}(\Sigma^{-q}(b^* U))=0$ since $\Sigma^{-q}(b^* U))=0\in H^{n}(V,\partial V,\mz_2)$.
\end{proof}

\textbf{\textsl{Proof of the odd case II:}} When $v_{n+1}(W)\neq0$, in our case $Y\times \mr^q\subset W=D(-L)$, by proposition 1.3, we see $H_*(X,\mz)$ is generated by the toric submanifolds which are certainly oriented and $W=D(-L_Y)$ is the disk bundle over $X$, whose homology group is also generated by these toric submanifolds. Then all the conditions of theorem 3.11 are satisfied. Thus, there exists an embedding sphere $g:S^n \hookrightarrow Y$ such that $g_*[S^n]=0$ and the normal bundle $\eta_g\cong TS^n$.

By the technique 2 in section 2 and lemma 3.8, we have topological decomposition:
\[
Y\cong Y' \sharp \ s (S^n \times S^n)
\]
where $b_n (Y')=0 $. Then we finish the proof of theorem 1.4.

\section{Even case}
\subsection{Intersection form and signature}
Let $X^{n+1}$ be a projective toric manifold with complex dimension $n+1$, $n>2$, even, and $i:Y\hookrightarrow X$ be a hypersurface of $X$. Since $n$ is even, the $n$-th homology group $H_n (Y,\mz)$ admits a unimodular symmetric intersection form:
\[
\begin{CD}
H_n(Y,\mz) \otimes H_n(Y,\mz)@>\cdot>>\mz
\end{CD}
\]

Since $(X,Y)$ is $n$-connected and $H_{odd}(X,\mz)=0$, like the odd case, we have
\[
\begin{CD}
0 @>>>H_{n+1}(X,Y)@>>> H_n (Y) @> i_* >> H_n(X)  @>>>0\\
@.      @AA\cong A        @AA h_Y A     @AAA    @.\\
   @. \pi_{n+1} (X,Y) @>>> \pi_n (Y) @>\pi_n(i)>> \pi_n (X) @.
\end{CD}
\]
The \textbf{vanishing cycles} $Ker(i_*)\subset H_n(Y,\mz)$ is what we mainly concerned, because:

\begin{lem}
Each element $\alpha\in Ker(i_*)$ can be represented by an embedding $f_{\alpha}:S^n \hookrightarrow Y$ such that $f_{\alpha}[S^n]=\alpha$ and the normal bundle $\eta_{f_{\alpha}}$ of $f_{\alpha}$ is stable trivial.
\end{lem}
\begin{proof}
Since $\pi_n (X,Y)\cong H_n (X,Y,\mz)\cong Ker(i_*)$, we see for each element $\alpha\in Ker(i_*)$, there exists an embedding $f_{\alpha}$ representing $\alpha$ and $\pi_n(i)(f_{\alpha})=0\in \pi_n(X)$.

The proof of the stable triviality of the normal bundle $\eta_{f_{\alpha}}$ is similar to the proof in lemma 3.8.
\end{proof}

When we restrict the intersection form of $H_n(Y,\mz)$ on $Ker(i_*)$, we get:
\begin{prop}
The intersection form on $Ker(i_*)$ is of type even, i.e. for any $\alpha\in Ker(i_*)$, $\alpha \cdot \alpha$ is even.
\end{prop}
\begin{proof}
For any $\alpha\in Ker(i_*)$, by lemma 4.1., we can use an embedding $f_{\alpha}$ to represent it. It is known that $\alpha\cdot \alpha=<e(\eta_{f_{\alpha}}),[S^n]>$, where $e(\eta_{f_{\alpha}})$ is the Euler class of the normal bundle $\eta_{f_{\alpha}}$. Furthermore, $<e(\eta_{f_{\alpha}}),[S^n]>$ is even if and only if the $n$-th Stiefel-Whitney class $w_n(\eta_{f_{\alpha}})$ of $\eta_{f_{\alpha}}$ is zero and this is just proved in lemma 4.1.
\end{proof}

The intersection pair on $H_n (Y,\mz)$ is equivalent to the cup product on $H^n (Y,\mz)$
\[
H^n(Y,\mz)\otimes H^n(Y,\mz)\longrightarrow \mz
\]
\[
(\alpha,\beta)\mapsto <\alpha \cup \beta,[Y]>
\]
through the Poincar\'{e} duality $PD:H^n(Y,\mz) \longrightarrow H_n (Y,\mz)$ and we also have exact sequence:
\[
\begin{CD}
0@>>>H^n (X,\mz) @>i^*>> H^n (Y,\mz) @>>> H^{n+1}(X,Y) @>>>0
\end{CD}
\]
We see the intersection form $(Ker(i_*),\cdot)$ is equivalent to $(PD^{-1}(Ker(i_*)),\cup)$ and the reason why we use the language of cohomolgy instead of homology is:
\begin{lem}
$PD^{-1}(Ker(i_*))=(i^* H^n (X))^{\perp}$
\end{lem}
\begin{proof}
For any $\alpha \in Ker(i_*)$ and $\beta\in H^n (X,\mz)$, we have:
\[
<PD^{-1}(\alpha)\cup i^* \beta,[Y]>=<i^*\beta,\alpha>=<\beta,i_* \alpha>=0
\]
we get $PD^{-1}(Ker(i_*))\subset(i^* H^n (X))^{\perp}$.

On the other hand, for any $PD^{-1}(\gamma)\in (i^* H^n (X))^{\perp}$, we see
\[
<PD^{-1}(\gamma)\cup i^* H^n (X,\mz),[Y]>=<i^*H^n (X,\mz),\gamma>=<H^n(X,\mz),i_*\gamma >=0
\]
Since $H_n(X,\mz)$ and $H^n(X,\mz)$ are free Abel groups, we get $i_*\gamma=0$.
\end{proof}

Next, we want to discuss some limit estimates about the $n$-th Betti number and the signature of the pair
$(H^n (Y_d, \mz),\cup)$. Recall that $i_d:Y_d\hookrightarrow X^{n+1}$ is the hypersurface of the toric manifold $X$ with $(i_d)_* [Y_d]=d(i_*[Y])$ and $degY_d=<\alpha_{Y_d}^{n+1},[X]>=d^{n+1}degY$, where $\alpha_{Y_d}\cap [X]=(i_d)_*[Y_d]$. We have proposition:
\begin{prop} We have limits:
\[
\lim_{d\rightarrow +\infty}\frac{b_n(Y_d)}{degY_d}=\lim_{d\rightarrow +\infty}\frac{b_n(Y_d)}{d^{n+1}degY}=1
\]
\[
0<\lim_{d\rightarrow +\infty}\frac{|sign(Y_d)|}{b_n(Y_d)}=2^{n+2}(2^{n+2}-1)\frac{B_{\frac{n+2}{2}}}{(n+1)!}<1
\]
\end{prop}
\begin{proof}
For the first limit, we know the Euler number $\chi(Y_d)$ of $Y_d$ equals $b_n(Y_d)+2\sum_{j=1}^{n-1}(-1)^{j}b_j(X) $
and
\begin{equation*}
\begin{split}
 \chi(Y_d)& =  \  <c_n(Y_d),[Y_d]>=<\frac{c(TX)}{1+d\alpha_{Y}},[X]>\\
          & =  \ d^{n+1}<\alpha_{Y}^{n+1},[X]> + O(d^n)
\end{split}
\end{equation*}
here $c(TX)$ and $c_n(Y_d)$ are the Chern classes. We have:
\[
\lim_{d\rightarrow +\infty}\frac{\chi(Y_d)}{d^{n+1}}=\lim_{d\rightarrow +\infty}\frac{b_n(Y_d)}{d^{n+1}}=degY\]
and we get:
\[
\lim_{d\rightarrow +\infty}\frac{b_n(Y_d)}{degY_d}=1
\]

For the second limit, we have identity:
\[
sign(Y_d)=<\tanh(d\alpha_Y)L(X),[X]>
\]
where $L(X)=L_1(X)+L_2(X)+\cdot\cdot\cdot$ is the $L$-class of $X$ and $\tanh(d\alpha_Y)=\sum_{j=1}^{+\infty}(-1)^{j-1}2^{2j}(2^{2j}-1)\frac{B_j}{(2j)!}(d\alpha_Y)^{2j-1}$.
Observe that:
\[
sign(Y_d)= (-1)^{\frac{n}{2}}2^{n+2}(2^{n+2}-1)\frac{B_{\frac{n+2}{2}}}{(n+2)!}d^{n+1}degY + O(d^n)
\]
we have limit:
\[
\lim_{d\rightarrow +\infty}\frac{|sign(Y_d)|}{b_n(Y_d)}=2^{n+2}(2^{n+2}-1)\frac{B_{\frac{n+2}{2}}}{(n+1)!}
\]
Furthermore, when $j>1$, we see:
\[
1+\frac{1}{2^{2j}}+\frac{1}{3^{3j}}+\cdot\cdot\cdot=\frac{B_j (2\pi)^{2j}}{2(2j)!}< \frac{\pi^2}{6}
\]
\[
\frac{B_j 2^{2j}(2^{2j}-1)}{(2j)!}<\frac{\pi^2}{3}
\frac{2^{2j}(2^{2j}-1)}{(2\pi)^{2j}}<\frac{\pi^2}{3}\frac{4^j}{\pi^{2j}}<1
\]

\end{proof}
\begin{coro}
$\lim_{d\rightarrow +\infty}b_n(Y_d)=+\infty$ and $\lim_{d\rightarrow +\infty}sign(Y_d)=+\infty$. When $d$ is big enough, $(H_n(Y_d,\mz),\cdot)$ is indefinite.
\end{coro}

\subsection{Proof of the even case}
Let $(H,<,>)$ be a unimodular symmetric bilinear form over $\mz$ and $F$ be a nonzero subgroup of $H$ such that $H/F$ is free and the map $F\longrightarrow Hom(F,\mz)$ induced by $<,>$ is injective. Denote $E=F^{\perp}:=\{x\in H|<x,F>=0\}$, we have:
\begin{thm}
If $rankH\geqslant Max\{4rankF, 2rankF+5\}$, then $E$ admits an orthogonal decomposition:
$$(E,<,>)\cong(A,<,>)\oplus (\oplus_{i=1}^{s}(\mz x_i\oplus \mz y_i,<,>))$$
where the intersection matrix of $\mz x_i\oplus \mz y_i$ is $
\left(\begin{array}{cc}
0 & 1\\
1 & c_i
\end{array}
\right)$, $c_i=0$ or $1$. For $(A,<,>)$, there are two possibilities:\\
(1). $(A,<,>)$ is definite and $rankA\geqslant max\{3rankF,rankF+5\}$\\
(2). $rank A<max\{3rankF,rankF+5\}$
\end{thm}
We'll prove the even case first and the proof of this theorem will be given in the next subsection.

\textbf{\textsl{Proof of the even case: }} \textbf{Step 1:} For the bilinear symmetric space $(H^n(Y_d,\mz),\cup)$, we know $PD^{-1}(Ker((i_d)_*))=(i^* H^n (X))^{\perp}$.
We want to show the injectivity of the map $H^n(X,\mz)\longrightarrow Hom(H^n(X,\mz),\mz)$ induced by the cup product in $H^n (Y_d,\mz)$.

Since $Y_d$ is the hypersurface of $X$, the hard Lefschetz theorem (\cite{Voisin}) tell us that the cohomology element $\alpha_{Y_d}$ representing $Y_d$ induces an injective map:
\[
\cup \alpha_{Y_d}: H^n(X,\mz)\longrightarrow H^{n+2}(X,\mz)
\]
For $i^*H^n(X,\mz)\subset H^n(Y_d,\mz)$, we have diagram:
\[
\begin{CD}
H^n(X,\mz) @>\cup \alpha_{Y_d} >> H^{n+2}(X,\mz) @>\cong  >> Hom(H^n (X,\mz),\mz)\\
 @Vi_d^* VV        @.     @AAA\\
i_d^*H^n (X,\mz) @>>>  H^n (Y_d,\mz) @>\cong>>  Hom(H^n(Y_d,\mz),\mz)
\end{CD}
\]
indeed, for any $x,y\in H^n(X,\mz)$, $x(y)=<i^*x\cup i^* y,[Y_d]>=<x\cup y\cup\alpha_{Y_d},[X]>=(x\cup\alpha_{Y_d})(y)$.

Furthermore, we see the restriction of $(H^n(Y_d,\mz),\cup)$ to $H^n(X,\mz)$ is just the bilinear form defined in theorem 1.5.

Thus we get a pair $(H_d,\cup)=(H^n(Y_d,\mz),\cup)$ with a free subgroup $F:=i^* H^n (X,\mz)$ such that \\
(1). $F^{\perp}=PD^{-1}(Ker(i_d)_*)=E_d$ with even type (proposition 4.2, lemma 4.3).\\
(2). If $d$ is big enough, $rank H_d>Max\{4rankF, 2rankF+5\} $ (proposition 4.4)

\textbf{Step 2:} By the algebraic decomposition theorem 4.6,
\[
(Ker(i_d)_*,\cdot)\cong (E_d,\cup)\cong A_d\oplus (\oplus_{i=1}^{s_d}(\mz x_i\oplus \mz y_i,<,>))
\]
where the intersection matrix of $\mz x_i\oplus \mz y_i$ is $
\left(\begin{array}{cc}
0 & 1\\
1 & c_i
\end{array}
\right)$, $c_i=0$ or $1$

By proposition 4.2, $Ker((i_d)_*)$ is of type even, $c_i$ must be zero.
Since $\lim_{d\rightarrow \infty}|signH^n(Y_d,\mz)|=+\infty$, when $d$ is big enough, the possibility (2) of theorem 4.6 can not happen, and $A_d$ is definite.

\textbf{Step 3:} By the process of removing handles of the even case in section 2, we see
\[
Ker(i_d)_*=A_d\oplus \oplus_{i=1}^{s_d} (\mz x_i \oplus \mz y_i)
\]
where the intersection matrix of $\oplus_{i=1}^{s_d} (\mz x_i \oplus \mz y_i)$ is $\oplus_{s_d}\left(\begin{array}{cc}
0 & 1\\
1 & 0
\end{array}
\right)$ and we get:
\[
Y_d\cong Y_d'\sharp s_d (S^n\times S^n)
\]

Since $A_d$ is definite and $sign(\mz x_i\oplus \mz y_i)=0$, we get identiy $2s_d=rank( Ker(i_d)_*) -|sign(Ker(i_d)_*)|$. Also, since $Ker(i_d)_*=(H^n(X,\mz))^{\perp}\subset H^n(Y_d,\mz)$ and the restriction of $(H^n(Y_d,\mz),\cup)$ to $H^n(X,\mz)$ is just the bilinear form defined in theorem 1.5, we get
\[
2s_d=b_n(Y_d)-b_n(X)-|sign(Y_d)-sign(H^n (X))|
\]
For the limit estimate, we have:
\[
\lim_{d\rightarrow +\infty} \frac{2s_d}{degY_d}=\lim_{d\rightarrow +\infty} \frac{2s_d}{b_n(Y_d)}=1-\lim_{d\rightarrow +\infty} \frac{|sign(Y_d)|}{b_n(Y_d)}
\]
\subsection{Proof of the algebraic decomposition theorem}

In order to prove theorem 4.6., we need some lemmas.

\begin{lem} Assume $E$ satisfy $rankE\geqslant 3rankF$, we can choose a basis $\{f_1,f_2,\cdot\cdot\cdot, f_{r+h}\}$ of $Hom(E,\mz)$ such that $\oplus_{i=1}^{r}\mz f_i\rightarrow Hom(E,\mz)/E$ is surjective, $r \leqslant rankF$, and $\oplus_{j=1}^{h} \mz f_{r+j}\subset E \subset Hom(E,\mz)$, $h\geqslant 2r$.
\end{lem}
\begin{proof}
First, since $H/F$ is free, we have:
\[
\begin{CD}
0 @>>>E @> {<,>} >>Hom(H,\mz) @>>> Hom(F,\mz) @>>> 0\\
@. @| @VVV @VVV @.\\
0 @>>>E @> {<,>} >>Hom(E,\mz) @>>> Hom(E,\mz)/E @>>>0
\end{CD}
\]
Since $H/E\cong Hom(H,\mz)/E\cong Hom(F,\mz)$, we see $Hom(H,\mz)\longrightarrow Hom(E,\mz)$ is surjective and $Hom(F,\mz)\longrightarrow Hom(E,\mz)/E$ is also surjective.

Second, $rankHom(F,\mz)=rankF$, we can choose $rankF$ elements $\{g_1,g_2 \cdot\cdot\cdot\}$ of $Hom(E,\mz)$ such that $\mz g_1+\mz g_2+\cdot\cdot\cdot+\mz g_{rankF}\longrightarrow Hom(E,\mz)/E$ is surjective. Then there is a subgroup $\mz g_1+\mz g_2+\cdot\cdot\cdot+\mz g_{rankF}\subset N\subset Hom(E,\mz)$ with $Hom(E,\mz)/N$ free and $r=rank N=rank (\mz g_1+\mz g_2+\cdot\cdot\cdot+\mz g_{rankF})\leqslant rankF$.

Third, Since $Hom(E,\mz)/N$ is free, let $\{f_1,\cdot\cdot\cdot,f_r\}$ be a basis of $N$ and extend it to a basis $\{f_1,\cdot\cdot\cdot,f_r,f_{r+1}',\cdot\cdot\cdot, f_{r+h}'\}$ of $Hom(E,\mz)$.
We know $N\longrightarrow Hom(E,\mz)/E$ is surjective, then for any $f_{r+i}',$ we can find $f_{r+i}=f_{r+i}'-\sum_{j=1}^{r} a_{ij}f_j$ with $[\sum_{j=1}^{r} a_{ij}f_j]=[f_{r+i}']\in Hom(E,\mz)/E$, we see $f_{r+i}\in E$.

Thus we obtain a basis $\{f_1\cdot\cdot\cdot f_r,f_{r+1},\cdot\cdot\cdot f_{r+h}\}$ of $Hom(E,\mz)$ such that $\{f_{r+1},\cdot\cdot\cdot f_{r+h}\}\subset E\subset Hom(E,\mz)$.
\end{proof}

\begin{lem}
Assume $E$ is indefinite and $rankE\geqslant Max\{3rankF,rankF+5\}$, we can find two elements $x,y\in E$ with $<x,x>=0,\ <x,y>=1, <y,y>=0$ or 1.
\end{lem}
\begin{proof}
First, by the lemma above, we have a basis $\{f_1\cdot\cdot\cdot f_r,f_{r+1},\cdot\cdot\cdot f_{r+h}\}$ of $Hom(E,\mz)$ with $\mz f_{r+1}\oplus \mz f_{r+2}\oplus \cdot\cdot\cdot \oplus \mz f_{r+h}\subset E\subset Hom(E,\mz), \ h>max\{2r,5\}$. Let $\{f_1^*,\cdot\cdot\cdot f_{r+s}^{*}\}$ be the dual of the basis $\{f_1,\cdot\cdot\cdot f_{r+h}\}$ in $Hom(E,\mz)^*=E$. Define:
\[
D=\mz f_{r+1}^{*}\oplus \cdot\cdot\cdot \mz f_{r+h}^*\subset E=Hom(E,\mz)^*
\]

Second, when $D$ is indefinite, since $rankD\geqslant 5,$ it is known from Meyer's theorem that there exists an indivisible element $x\in D$ such that $<x,x>=0$ (\cite{KW}, p255). Then we can also choose an element $y'\in \mz f_{r+1}\oplus \cdot\cdot\cdot \oplus\mz f_{r+h}\subset E$ such that $<x,y'>=1$. Let $y=y' - [\frac{<y',y'}{2}]x'$, we have
$<x,x>=0,<x,y>=1,<y,y>=0$ or 1.

Third, when $D$ happens to be definite, define $D':= \mz (f_{r+1}^{*}-c_1 f_1^*)\oplus \mz(f_{r+2}^* - c_2 f_1^*)\oplus \mz(f_{r+3}^*-c_3 f_{2}^{*})\oplus\mz(f_{r+4}^*-c_4 f_{2}^{*}) \oplus\cdot\cdot\cdot\oplus\mz (f_{3r-1}^{*}-c_{2r-1}f_{r}^{*})\oplus \mz (f_{3r}^{*}-c_{2r} f_{r}^{*})\oplus\cdot\cdot\cdot\oplus \mz(f_{r+h}^{*}-c_h f_r^*)$.

If we can choose proper $\{c_i\in \mz\}$ to make $D'$ indefinite, we can still find an indivisible element $x\in D'$ such that $<x,x>=0$, and we can also find $y\in \oplus_{j=1}^{h}\mz f_{r+j}\in E$ such that $<x,y>=1$. So, all we need to do is to prove the next lemma.
\end{proof}

\begin{lem}
Following lemma 4.8, suppose $D$ is definite, we can choose proper $\{c_i\in \mz\}$ to make $D'$ indefinite.
\end{lem}
\begin{proof}
Assume $D$ is positive definite under $<,>$. Consider the real space $\overline{E}:= E\otimes \mr, \ \overline{D}:=D\otimes \mr$, $\overline{D'}:=D'\otimes \mr$ and let $\{f_{1}^{*},\cdot\cdot\cdot,f_{r+h}^*\}$ be the Euclidean orthogonal standard basis of $\overline{E}$. Define:
\[
F: \overline{E}\longrightarrow \mr
\]
\[
\sum a_i f_i^* \mapsto \sum a_i a_j <f_i^*,f_j^*>
\]
Observe that $F$ is just the the extension of the map $:E\rightarrow \mz,x\mapsto <x,x>$ to $\overline{E}$.

Note that $E$ is indefinite and $\mq$-uninodular under $<,>$, since by assumption $H$ is unimodular and $F$ is $\mq$-unimodular. Then we can find a $v\in \overline{E}$ such that $F(v)<0$ and the Euclidean norm $|v|=1$, i.e. $v=\sum_{i=1}^{r}a_i f_i^* + \sum_{j=1}^{h} b_j f_{r+j}^*$, $\sum a_i^2 + \sum b_j^2 =1.$ Since $D$ is definite, we see $F(\overline{D}-\{0\})>0$ and $(a_1,\cdot\cdot\cdot,a_r)\neq (0,\cdot\cdot\cdot,0)$.

In the Euclidean norm with orthogonal standard basis $\{f_i^*\}$, we have decomposition $\overline{E}=\overline{D'}\oplus (\overline{D'})^{\perp}$. By calculation, we see $(\overline{D'})^{\perp}$ has a standard orthogonal basis:
\[
(\overline{D'})^{\perp}=span\{\frac{f_{1}^* + c_1 f_{r+1}^* +c_2 f_{r+2}^*}{\sqrt{1+c_1^2 +c_2^2}} ,\frac{f_{2}^* + c_3 f_{r+3}^* +c_4 f_{r+4}^*}{\sqrt{1+c_3^2 +c_4^2}}
\]
\[
\cdot\cdot\cdot, \frac{f_{r-1}^* + c_{2r-3} f_{3r-3}^* +c_{2r-2} f_{3r-2}^*}{\sqrt{1+c_{2r-3}^2 +c_{2r-2}^2}},\frac{f_{r}^* + c_{2r-1} f_{3r-1}^* +c_{2r} f_{3r}^*+\cdot\cdot\cdot, c_h f_{r+s}^*}{\sqrt{1+c_{2r-1}^2 +\cdot\cdot\cdot + c_h^2}} \}
\]
For convenience, denote this basis by $\{g_1,g_2,\cdot\cdot\cdot,g_r\}$.

For the vector $v=\sum_{i=1}^{r}a_i f_i^* + \sum_{j=1}^{h} b_j f_{r+j}^*$, we can decompose $v=v_1 + v_2$ such that $v_1\in \overline{D'}$ and $v_2\in (\overline{D'})^{\perp}.$
By calculation,
\[
v_2=\frac{a_{1} + c_1 b_{r+1} +c_2 b_{r+2}}{\sqrt{1+c_1^2 +c_2^2}}g_1 +\cdot\cdot\cdot+\frac{a_{r-1} + c_{2r-3} b_{3r-3} +c_{2r-2} b_{3r-2}}{\sqrt{1+c_{2r-3}^2 +c_{2r-2}^2}}g_{r-1}
\]
\[
+ \frac{a_{r} + c_{2r-1} b_{3r-1} +c_{2r} b_{3r}+\cdot\cdot\cdot, c_h b_{r+h}}{\sqrt{1+c_{2r-1}^2 +\cdot\cdot\cdot + c_h^2}}g_r
\]
Since $\sum a_i^2 + \sum b_j^2 =1,$ for $\forall \epsilon >0$, we can choose proper $c_i\in \mz$ (\cite{KW}, p256) such that
\[
|\frac{a_{1} + c_1 b_{r+1} +c_2 b_{r+2}}{\sqrt{1+c_1^2 +c_2^2}}|<\frac{\epsilon}{r},\cdot\cdot\cdot,|\frac{a_{r-1} + c_{2r-3} b_{3r-3} +c_{2r-2} b_{3r-2}}{\sqrt{1+c_{2r-3}^2 +c_{2r-2}^2}}|<\frac{\epsilon}{r}
\]
\[
|\frac{a_{r} + c_{2r-1} b_{3r-1} +c_{2r} b_{3r}+\cdot\cdot\cdot, c_h b_{r+h}}{\sqrt{1+c_{2r-1}^2 +\cdot\cdot\cdot + c_h^2}}|<\frac{\epsilon}{r}
\]

The function $F$ is continuous and $F(v)<0$, if the Euclidean norm of $v_2=v-v_1$ is small enough, then the element $v_1\in \overline{D'}$ satisfy $F(v_1)<0$. Furthermore, $D'$ is not negative definite, since $D$ is positive and $rank D'=rank D, 2rankD>rank E=rank D + rankF$, thus we see $D'$ is indefinite.

\end{proof}

\textbf{\textsl{Proof of theorem 4.6.:}} We use induction on $rankH$, since $rankH\geqslant Max \{4rankF,\  2rankF+5\}$, we get $rankE\geqslant Max\{3rank F, rankF+5\}$. If $E$ is definite, we've done. If $E$ is indefinite, then by the lemmas we've just proved, there exist two elements $x,y\in E$ such that $<x,x>=0,<x,y>=1,<y,y>=0$ or 1.

We get orthogonal decomposition under $<,>$:
\[
H=H'\oplus (\mz x \oplus \mz y), \ E=E'\oplus (\mz x \oplus \mz y)
\]
Observe that $F\subset H'$ and $E\cap H'=E'=F^{\perp}\subset H'$, by the induction, we've finished our proof.


\vspace{1cm}
\noindent{Wei WANG}\\
School of Mathematical Sciences,\\
Fudan University,\\
Shanghai 200433, P.R. China.\\
Email: weiwang@amss.ac.cn

\end{document}